\documentclass{amsart}
\usepackage{hyperref}
\usepackage{comment} 

\usepackage[usenames]{color}
\usepackage[all]{xy}
\usepackage[notcite,notref]{}

\usepackage{amsfonts,amsmath,amssymb,amsthm,amscd,amsxtra}
\usepackage{enumerate,epsfig,verbatim}

\usepackage{enumerate,verbatim}
\usepackage{centernot}
\usepackage{todonotes}

\usepackage[T1]{fontenc}
\usepackage[usenames,dvipsnames]{pstricks}
\usepackage[mathscr]{eucal}
\usepackage[notcite,notref]{}
\usepackage{tikz}
\usetikzlibrary{matrix,quotes}
\usepackage[notcite, notref]{}

\usepackage[usenames,dvipsnames]{pstricks}
\usepackage[mathscr]{eucal}
\usepackage{amsfonts,amsmath,amssymb,amsthm,amscd,amsxtra}
\usepackage{tikz-cd} 
\usepackage{enumerate,verbatim}
\usepackage{mathptmx}
\usepackage[notcite, notref]{}

\usepackage{mathtools}
\usepackage{nicefrac}
\usepackage{array}
\usepackage{xcolor}

\usepackage{amsfonts}
\usepackage[margin=1.4in]{geometry}
\usepackage{color}
\usepackage[notcite,notref]{}
\usepackage{url}
\usepackage{datetime}
\usepackage{mathtools}

\usepackage{amssymb}
\usepackage{amsmath}
\usepackage{amsfonts}
\usepackage{amsmath}
\usepackage{amsthm}
\usepackage{amssymb}
\usepackage{amscd}
\usepackage{amsfonts}
\usepackage{amsxtra} 

\usepackage{epsfig}
\usepackage{verbatim}

\SelectTips{cm}{}

\usepackage{amsfonts,amsmath,amssymb,amsthm,amscd,amsxtra}
\usepackage{enumerate,epsfig,verbatim}

\usepackage{amssymb,amstext,amsmath,amscd,amsthm,amsfonts,enumerate,graphicx,latexsym}
\usepackage[usenames]{color}

\usepackage{nicefrac}
\usepackage{array}

\usepackage{comment}
\usepackage{amssymb,amstext,amsmath,amscd,amsthm,amsfonts,enumerate,graphicx,latexsym,color,stmaryrd}
\usepackage{mathptmx}
\usepackage[all]{xy}
\usepackage{comment}
\newcommand{\com}[1]{{\color[rgb]{0.0, 0.5, 0.0} #1}}
\newcommand{\old}[1]{{\color{red} #1}}
\newcommand{\new}[1]{{\color{blue} #1}}
\newtheorem{thm}{Theorem}[section]
\newtheorem{prop}[thm]{Proposition}
\newtheorem{cor}[thm]{Corollary}

\theoremstyle{definition}
\newtheorem{chunk}[thm]{\hspace*{-1.065ex}\bf}
\newtheorem{lem}[thm]{Lemma}
\newtheorem{dfn}[thm]{Definition}

\newtheorem{eg}[thm]{Example}

\newtheorem{ques}[thm]{Question}

\newtheorem{rem}[thm]{Remark}

\theoremstyle{remark}
\newtheorem{claim}{Claim}
\newtheorem*{claim*}{Claim}
\numberwithin{equation}{thm}

\newcommand{\X}{\mathbb{X}}

\newcommand{\cG}{\mathcal{G}}

\newcommand{\fm}{\mathfrak{m}}

\newcommand{\gcx}{\operatorname{gcx}} 
\newcommand{\bR}{\mathbb{R}}
\newcommand{\bN}{\mathbb{N}}
\newcommand{\cC}{\mathcal{C}}
\newcommand{\cY}{\mathcal{Y}}
\newcommand{\End}{\operatorname{End}} 

\def \Tr {\operatorname {Tr}}

\newcommand{\cone}{\textup{cone}}

 \DeclareMathOperator{\HH}{H}

\newcommand{\m}{\mathfrak{m}}
\newcommand{\fn}{\mathfrak{n}}

\newcommand{\comp}{\mathbb{C}}

\def\fm{\mathfrak{m}}
\def\fp{\mathfrak{p}}
\def \Syz {\operatorname{Syz}}

\def\cX{\mathcal{X}}
  
\def \codim {\operatorname{codim}}

\def\rB{\mathrm{B}}

\def \lX{\mathcal X} 
\def \Y {\mathcal Y}

\def\depth{\operatorname{\mathrm{depth}}}

\def\embdim{\operatorname{\mathrm{embdim}}}
\def\Ext{\operatorname{\mathrm{Ext}}}

\def\Hom{\operatorname{\mathrm{Hom}}}
\def\limsup{\operatorname{\mathrm{limsup}}}

\def\mod{\operatorname{\mathrm{mod}}}
\def\qpd{\operatorname{\mathrm{qpd}}}
\def\Tor{\operatorname{\mathrm{Tor}}}

\def\syz{\Omega}
\def \TF {\operatorname{TF}}  
\def\Spec{\operatorname{\mathrm{Spec}}}

\def \cdim {\operatorname{\mathsf{CI-dim}}}

\newcommand{\RHom}{\operatorname{{\bf{R}}Hom}}

\def\ind{\operatorname{\mathsf{i}}}

\def\Ho{\operatorname{\mathsf{H}}}

\DeclareMathOperator{\CIdim}{\operatorname{\mathsf{CI-dim}}}
\DeclareMathOperator{\Hdim}{\mathsf{H-dim}}
\DeclareMathOperator{\Gdim}{\mathsf{G-dim}}
\DeclareMathOperator{\pd}{\mathsf{pd}}
\DeclareMathOperator{\cx}{\mathsf{cx}}
\DeclareMathOperator{\px}{\mathsf{px}} 
\DeclareMathOperator{\curv}{\mathsf{curv}}
\DeclareMathOperator{\rGdim}{\operatorname{\mathsf{red-G-dim}}}

\DeclareMathOperator{\rpdim}{\operatorname{\mathsf{red-pd}}}

\DeclareMathOperator{\rHdim}{\operatorname{\mathsf{red-H-dim}}}
\DeclareMathOperator{\rCIdim}{\operatorname{\mathsf{red-CI-dim}}}
\DeclareMathOperator{\id}{\operatorname{\mathsf{id}}}

\begin{document} 
\thanks{2020 {\em Mathematics Subject Classification.}13D07, 13D05, 13C13, 13C14, 13C60, 13H10}  
\thanks{{\em Key words and phrases.} complete intersection, Gorenstein ring, (reducible) complexity, (reducing) Gorenstein dimension, (reducing) projective dimension, syzygy, totally reflexive module}   

\baselineskip=15pt
\baselineskip=15pt
\title[Some characterizations of local rings via reducing dimensions ]{Some characterizations of local rings via reducing dimensions}  

\author{Olgur Celikbas}
\address{Olgur Celikbas\\ School of Mathematical and Data Sciences, West Virginia University
Morgantown, WV 26506 U.S.A}
\email{olgur.celikbas@math.wvu.edu}

\author{Souvik Dey}
\address{Souvik Dey\\ Department of Mathematics, University of Kansas, 405 Snow Hall, 1460 Jayhawk Blvd.,
Lawrence, KS 66045, U.S.A.}
\email{souvik@ku.edu}

\author{Toshinori Kobayashi}
\address{Toshinori Kobayashi \\ School of Science and Technology, Meiji University, 1-1-1 Higashi-Mita, Tama-ku, Kawasaki-shi, Kanagawa 214-8571, Japan}
\email{toshinorikobayashi@icloud.com}

\author[Hiroki Matsui]{Hiroki Matsui}
\address{Hiroki Matsui\\ Department of Mathematical Sciences,
Faculty of Science and Technology,
Tokushima University,
2-1 Minamijosanjima-cho, Tokushima 770-8506, JAPAN}
\email{hmatsui@tokushima-u.ac.jp}

\thanks{Kobayashi was partly supported by JSPS Grant-in-Aid for JSPS Fellows 21J00567. Matsui was partly supported by JSPS Grant-in-Aid for Early-Career Scientists 22K13894} 

\maketitle

\begin{abstract}
In this paper we study homological dimensions of finitely generated modules over commutative Noetherian local rings, called reducing homological dimensions. We obtain new characterizations of Gorenstein and complete intersection local rings via reducing homological dimensions. For example, we extend a classical result of Auslander and Bridger, and prove that a local ring is Gorenstein if and only if each finitely generated module over it has finite reducing Gorenstein dimension. Along the way, we prove various connections between complexity and reducing projective dimension of modules.  
\end{abstract}

\section{introduction}

Throughout all rings, usually denoted by $R$ or $S$, are assumed to be commutative, Noetherian, and local and all modules are assumed to be finitely generated. Denote by $\mod R$ the category of $R$-modules.  

Homological dimensions such as the projective dimension $\pd_R$, the Gorenstein dimension $\Gdim_R$, and the complete intersection dimension $\cdim_R$ are invariants that assign an element of $\bN \cup \{\infty, -\infty\}$ to an isomorphism class of $R$-modules. 
An $R$-module $M$ with $\pd_R(M) < \infty$ (resp. $\Gdim_R(M) < \infty$, resp. $\cdim_R(M)<\infty$) has the similar property with the modules over regular (resp. Gorenstein, resp. complete intersection) rings.
The most important property is the following theorem, which gives a characterization of local rings via such homological dimensions.

\begin{thm}\label{clas} \cite[Th\'eor\`eme 3]{Ser}\cite[(1.4.9)]{AGP} \cite[Theorem 4.20]{AB}
Let $(R,\fm,k)$ be a local ring.
The following are equivalent:
\begin{enumerate}[\rm(i)]
\item
$R$ is regular (resp. Gorenstein, resp. complete intersection)
\item
$\pd_R(M) < \infty$ (resp. $\Gdim_R(M) < \infty$, resp. $\cdim_R(M) < \infty$) for each $R$-module $M$.
\item
$\pd_R(k) < \infty$ (resp. $\Gdim_R(k) < \infty$, resp. $\cdim_R(k) < \infty$)
\end{enumerate}
\end{thm} 

Reducing homological dimensions were introduced by Araya and Celikbas \cite{CA}, and subsequently with a weaker condition by Araya and Takahashi \cite{at}. 
However finiteness of reducing homological dimensions is a quite weaker condition than finiteness of the corresponding homological dimensions, several known results for modules with finite homological dimensions has been generalized to modules with finite reducing homological dimensions; see \cite{CA, acck, at, redring}. 
Therefore, studying reducing homological dimensions is an important subject in commutative algebra.

Keeping in mind the important property Theorem \ref{clas} for homological dimensions, it is natural to ask that whether we can prove the similar result using reducing homological invariants instead of using homological dimensions.
As we will see in \ref{rprcid} that $\rpdim_R(M) < \infty$ if and only if $\rCIdim_R(M) < \infty$ for an $R$-module $M$, it suffices to consider the reducing projective dimension $\rpdim$ and the reducing Gorenstein dimension $\rGdim$. Thus we consider the following question.

\begin{ques} \label{qci} Let $R$ be a local ring. If $\rpdim_R(M) < \infty$ (resp. $\rGdim_R(M)< \infty$) for each $R$-module $M$, then must $R$ be a complete intersection (resp. Gorenstein)?
\end{ques}

For the reducing Gorenstein dimension, we obtain a complete answer to Question \ref{qci}, which generalizes both \cite[Corollary 3.2]{CA} and Theorem \ref{clas}.

\begin{thm} \rm{(Theorem \ref{Gdimcor})} \label{main0} Let $(R,\m,k)$ be a local ring. Then the following are equivalent:
\begin{enumerate}[\rm(i)]
\item $R$ is Gorenstein. 
\item  $\rGdim_R(M)<\infty$ for each $R$-module $M$.
\item There exists a resolving subcategory $\lX$ of $\mod R$ containing $k$  such that $\rGdim_R(M)<\infty$ for each $M\in \lX$. \qed 

\end{enumerate}
\end{thm}
\noindent
Here, 
a full subcategory $\lX$ of $\mod R$ is said to be {\it resolving} if it is closed under taking direct summands, extensions, kernels of epimorphisms, and containing $R$. The reason why a resolving subcategory appears in (iii) is that finiteness of reducing homological dimension is ill-behaved with respect to short exact sequences (see Example \ref{2.5} and Example \ref{2.5(i)} for instance). In fact, our result Theorem \ref{Gdimcor} contains more than Theorem \ref{main0} as it also shows that $R$ is Gorenstein if $\rGdim_R(\Tr_R \syz^n_R k)<\infty$ for some $n\ge \depth R$ (where $\Tr_R(-)$ denotes the Auslander-Bridger transpose).

It is known that each module over a local complete intersection ring of codimension $c$ has finite reducing projective dimension of at most $c$; see \cite{bergh}. In this paper we investigate whether or not the converse of this fact is true, and whether one can obtain a characterization of complete intersection property via reducing projective dimensions. For small $c$, we obtain the following result.  

\begin{thm}\rm{(Theorem \ref{mainci})}\label{main}  Let $(R,\m,k)$ be a local ring and let $c\le 2$. Then the following are equivalent:  
\begin{enumerate}[\rm(i)]
\item $R$ is a complete intersection of codimension at most $c$.
\item $\rpdim_R(M) \le c$ for each  $R$-module $M$. 
\item $R$ is Gorenstein and $\rpdim_R(k) \le c$. 
  \qed  
\end{enumerate}   
\end{thm}
\noindent
Similar to Theorem \ref{Gdimcor}, Theorem \ref{main} also contains a characterization in terms of $\rpdim_R(\Tr_R \syz^n_R k)$.    

We now briefly describe the structure of the paper. 

In Section \ref{sec 2}, we recall preliminary definitions, notations, and results related to reducing dimensions, and reducible complexity \cite{bergh}.   

In Section \ref{sec 3}, we prove Theorem \ref{Gdimcor} regarding characterization of Gorenstein rings via reducing Gorenstein dimension, part of which was highlighted  above. One of the main ingredients in the proof of this is Theorem \ref{cor-app} which states that if a module $M$ of finite reducing Gorenstein dimension belongs to a resolving subcategory $\lX$ such that $\lX$ contains all totally reflexive modules, and every module in $\lX$ has depth at least the depth of the ring, then $M\in \syz \lX$. We deduce various consequences of Theorem \ref{cor-app} even other than  Theorem \ref{Gdimcor}.


In Section \ref{sec 4}, we prove Theorem \ref{mainci}. Along the way, we  show in Corollary \ref{corcxrpd}, that for a module of finite complexity, the complexity is bounded above by the reducing projective dimension. The proof of this result, and that of the main result Theorem \ref{mainci} of this section heavily relies on a technical result, namely Proposition \ref{corcx}, which gives some connection between complexities of two modules $M$ and $K$ when they fit into an exact sequence of the form $0\to M^{\oplus a}\to K\to \syz^n_R M^{\oplus b}\to 0$.



\section{Definitions and examples}\label{sec 2}

Given a ring $R$, we denote by $(-)^*$ the $R$-dual functor $\Hom_R(-,R)$ of $R$-modules. For an $R$-module $M$ and an integer $i\ge 0$, we denote by $\syz^i_R M$ the $i$-th syzygy module in a minimal free resolution of $M$. Whenever the ring $R$ is clear from the context, we only write $\syz^i M$ in place of $\syz^i_R M$.  


First we recall the definition of complete intersection dimension and its related concepts. For the other homological dimension such as Gorenstein dimension $\Gdim$, we refer the reader to \cite{Gdimbook}.


\begin{chunk} \label{CI} (\textbf{Complete intersection dimension} \cite[1.2]{AGP}) Let $R$ be a ring and let $M$ be an $R$-module.
Define the {\it complete intersection dimension} of $M$ as
$$
\CIdim_R(M) := \inf \{\pd_S(M \otimes_R R') - \pd_S R' \mid R \to R' \twoheadleftarrow S \mbox{ is a quasi-deformation} \}.
$$
Here a diagram $R \to R' \twoheadleftarrow S$ of local ring maps is a {\it quasi-deformation} if $R \to R'$ is flat, $R' \twoheadleftarrow S$ is surjective such that its kernel is generated by an $S$-regular sequence.
Note that, if $R$ is a complete intersection ring, then it follows by definition that $\CIdim_R(M)<\infty$. 

It is shown in \cite[Theorem 1.4]{AGP} that there are inequalities
\begin{equation}\tag{\ref{CI}.1}
\Gdim_R(M) \le \CIdim_R(M) \le \pd_R(M)
\end{equation}
for any $R$-module $M$.
\end{chunk}

\begin{chunk} \label{rcx} (\textbf{Complexity and reducible complexity} \cite[4.2]{av} and \cite{bergh}) Let $R$ be a local ring and let $M$ be an $R$-module.
The \emph{complexity} $\cx_R(M)$ of $M$ is defined by the smallest integer $r\geq 0$ such that  there exists a real number $A$ with $\beta_n(M) \leq A \cdot n^{r-1}$ for all $n\gg 0$, where $\beta_n(M)$ denotes the $n th$ Betti number of $M$. (If there do not exist such integer $r$, we have $\cx_R(M)=\infty$ by convention). Note that $\cx_R(M)=0$ if and only if $\pd_R(M)<\infty$, and $\cx_R(M)\leq 1$ if and only if $M$ has bounded Betti numbers. 

The module $M$ is said to have {\it reducible complexity} if either $\pd_R(M)<\infty$, or $0<\cx_R(M)<\infty$ and there is an integer $r\geq 1$ and short exact sequences of $R$-modules
$$
\{0 \to K_i \to K_{i+1} \to \Omega^{n_i}_R K_i \to 0\}_{i=0}^{r}
$$
where $K_0 = M$, $\pd_R(K_r)<\infty$, and $\cx_R(K_{i+1}) < \cx_R(K_i)$ for each $i=0,1,\ldots, r-1$.

The original definition of reducible complexity \cite[Definition 2.1]{bergh} requires $\depth_R(K_i)=\depth_R(M)$ for each $i$. Note that this condition holds automatically when $R$ is Cohen-Macaulay (see \cite[Remark after Definition 2.1]{bergh}). However, as we do not need this depth condition for our arguments, we do not include it in the definition. 
\end{chunk}

\begin{chunk}(\cite[Proposition 2.2]{bergh})
If an $R$-module has finite complete intersection dimension, then $M$ has reducible complexity. 	\qed
\end{chunk}

Motivated by this definition, Araya-Celikbas introduced the following concept.


\begin{chunk} \label{rdim} (\textbf{Reducing homological dimensions} \cite[Definition 2.1]{CA} and \cite[Definition 2.5]{at}) Let $R$ be a local ring, $M$ be an $R$-module, and let $\Hdim$ be a \emph{homological dimension} of $R$-modules, for example, $\Hdim_R \in \{\pd_R, \CIdim_R, \Gdim_R\}$.  

We write $\rHdim_R(M)<\infty$ provided that 
\begin{enumerate}[\rm(a)]
\item $\Hdim_R(M) < \infty$, or
\item $\Hdim_R(M) = \infty$ and there exist integers $r,a_i,b_i\geq 1$, $n_i\geq 0$, and short exact sequences of $R$-modules of the form 
\begin{equation}\tag{\ref{rdim}.1}
0 \to K_{i-1}^{{\oplus a_i}} \to K_{i} \to \Omega^{n_i}_R K_{i-1}^{{\oplus b_i}} \to 0
\end{equation}
for each $i=1, \ldots r$, where $K_0=M$ and $\Hdim(K_r)<\infty$. 
In this case, we call $\{K_0,  \ldots, K_r\}$ a \emph{reducing $\Hdim$-sequence} of $M$.
\end{enumerate}

The \emph{reducing homological dimension} $\rHdim_R(M)$ of $M$ is defined as follows: If $\Hdim_R(M)=\infty$, we set
$$
\rHdim_R(M):=\inf\{ r\in \bN : \text{there is a reducing $\Hdim$-sequence }  K_0,  \ldots, K_r \text{ of }  M\}
$$
and we set $\rHdim_R(M):=0$ if $\Hdim_R(M)<\infty$. 
We note that our definition is more relaxed from that of \cite[2.1]{CA} and aligns with \cite[2.5]{at} in that we take $n_i\ge 0$ instead of $n_i>0$.

It follows from the inequalities (\ref{CI}.1) that there are inequalities
$$
\rGdim_R(M) \le \rCIdim_R(M) \le \rpdim_R(M)
$$
for any $R$-module $M$.\qed
\end{chunk}

Unlike homological dimensions, finiteness of reducing homological dimensions is not compatible with short exact sequences. The following two examples show that finiteness of reducing homological dimensions are not in general closed under extensions, kernels of epimorphisms or co-kernels of monomorphisms.
\begin{eg}\label{2.5}
Here we give a family of examples (of ring of every dimension) showing finiteness of $\rpdim$ or $\rGdim$ is not in general closed under taking extensions.

Let $(R,\m,k)$ be a local Cohen--Macaulay ring of dimension $d$, and minimal multiplicity, admitting a canonical module $\omega_R$, and assume $k$ is infinite. Then, we have $\m^2=(x_1,\cdots,x_d)\m$. Write $\overline R=R/(x_1,\cdots,x_d)$ which has maximal ideal $\m/(x_1,\cdots,x_d)$ and canonical module $\omega_{\overline R}\cong \omega_R/(x_1,\cdots,x_d)\omega_R$. We have a short exact sequence of $R$ and $\overline R$-modules 
\begin{align*}\label{example}
0\to \m \omega_{\overline R} \to \omega_{\overline R}\to \dfrac{\omega_{\overline R}}{\m \omega_{\overline R}}\to 0
\end{align*}
where the right most module is clearly a $k$-vector space, and the left most module is also a $k$-vector space since $\m^2=(x_1,\cdots,x_d)\m$. Hence both the left and right most modules have finite reducing projective dimension due to \cite[Theorem 1.2]{redring}. However, if $\rGdim_R(\omega_{\overline R})<\infty$, then by \cite[Proposition 3.9. and 3.13]{redring} we obtain $\rGdim_{\overline R}(\omega_{\overline R})\le \rGdim_R(\syz^d_R \omega_{\overline R})+d<\infty$. But then $\overline R$ is Gorenstein by  \cite[Corollary 3.3]{CA}, and hence $R$ is Gorenstein. Consequently, if $R$ is not Gorenstein, then $\rGdim_R(\omega_{\overline R})=\infty=\rpdim_R(\omega_{\overline R})$. Applying \ref{C1}(i) to this same exact sequence and using \cite[3.8 and 3.13]{redring}, we get an exact sequence showing finiteness of $\rpdim$ or $\rGdim$ is not in general closed under co-kernels of monomorphisms. 
\end{eg}

\begin{eg}\label{2.5(i)} Now we give a family of examples (of ring of every positive dimension) showing finiteness of $\rpdim$ or $\rGdim$ is not in general closed under kernels of epimorphisms. Let $(R,\m,k)$ be a local $1$-dimensional Cohen--Macaulay, non-Gorenstein ring admitting a canonical ideal $\omega_R$. Also assume $R$ has minimal multiplicity and  $\End_R(\m)$ is a Gorenstein ring. An explicit example of such a ring is $R=k[[t^e, t^{e+1},\cdots,t^{2e-1}]]$, where $e\ge 3$ is an integer (see \cite[Example 3.13]{gmp}). By \cite[Theorem 5.1]{gmp} and \cite[Theorem 1.4 and Corollary 1.5]{kob}, we get an exact sequence $0\to \omega_R \to \m \to k \to 0$. Now let $S:=R[[x_2,\cdots,x_d]]$, where $d\ge 1$ (if $d=1$, we interpret $S$ as $R$ itself). As $S$ is flat over $R$, we obtain an exact sequence $0\to S\otimes_R \omega_R \to S \otimes_R \m \to S \otimes_R k \to 0$. As $S/\m S\cong k[[x_2,\cdots,x_d]]$ is Gorenstein, so  by \cite[Theorem 3.3.14]{bh} we get $S\otimes_R \omega_R \cong \omega_S$. Hence the exact sequence becomes 
$$
0\to  \omega_S \to S \otimes_R \m \to S \otimes_R k \to 0.
$$  
Now as $R$ has minimal multiplicity, so $\rpdim_R k$ and $\rpdim_R \m$ are finite by \cite[Theorem 1.2 and Corollary 3.13]{redring}. Hence $\rpdim_S (S\otimes_R \m)$ and $\rpdim_S (S\otimes_R k)$ are finite by \cite[Corollary 3.2]{redring}. Now since $S/(x_2,\cdots,x_d)\cong R$ is not Gorenstein, so $S$ is not Gorenstein, and thus $\rpdim_S \omega_S=\rGdim_S \omega_S=\infty$ by \cite[Corollary 3.3]{CA}.  
\end{eg}

\begin{chunk}\label{rprcid}
For an $R$-module $M$, $\rCIdim_R(M)< \infty$ if and only if $\rpdim_R(M)< \infty$.
Indeed, the if part follows from the inequality $\rCIdim_R(M) \le \rpdim_R(M)$.  

Conversely, suppose $\rCIdim_R(M) < \infty$.
If $\rCIdim_R(R) = 0$ i.e., $\CIdim_R(M) < \infty$, then $M$ has reducible complexity and hence it has finite reducing projective dimension from the definition.
If $0<\rCIdim_R(R) < \infty$, then we can take a reducing $\CIdim$-sequence $\{K_0, \ldots, K_r\}$ of $M$.
Then $K_r$ has finite complete intersection dimension and hence it has a finite reducing projective dimension as above. Therefore, we get a reducing $\pd$-sequence of $M$ by connecting the reducing $\CIdim$-sequence $\{K_0, \ldots, K_r\}$ of $M$ and the reducing $\pd$-sequence of $K_r$.
This conclude that $M$ has finite reducing projective dimension.\qed
\end{chunk}

For the rest of this paper, we will concentrate on $\rpdim$ and $\rGdim$.
We will often use the following property from \cite[Proposition 3.8]{redring}.

\begin{chunk} Let $\Hdim=\pd$ or $\Gdim$. Then, for any $R$-module $M$ and free $R$-module $F$, we have $\rHdim(M)=\rHdim(M\oplus F)$. \qed
\end{chunk}


The following relations between the (reducible) complexity and the reducing projective dimension will be obtained in Corollary \ref{corcxrpd}.

\begin{chunk} \label{rcxrpd} Let $R$ be a local ring and let $M$ be an $R$-module. 
\begin{enumerate}[\rm(i)]
\item If $\cx_R(M) < \infty$, then it follows that $\cx_R(M) \le \rpdim_R(M)$.
\item If $M$ has reducible complexity, then $\cx_R(M) = \rpdim_R(M)$. \qed
\end{enumerate}
\end{chunk}


The reducing projective dimension of a module of finite complete intersection dimension is bounded by the codepth of the ring. More precisely, in view of \ref{rcxrpd}, one has:

\begin{chunk} Let $R$ be a local ring and let $M$ be an $R$-module. Assume $\CIdim_R(M)<\infty$. Then,
\begin{enumerate}[\rm(i)]
\item $M$ has reducible complexity; see \cite[Proposition 2.2]{bergh}
\item $\cx_R(M)=\rpdim_R(M)\leq \embdim(R)-\depth(R)$; this follows from part (i), \ref{rcxrpd}, and \cite[5.6]{AGP}.

\item If $\rpdim_R(M)=\embdim(R)-\depth(R)$, then $R$ is a complete intersection ring; this follows from (ii), and \cite[5.6]{AGP}.  \qed
\end{enumerate}
\end{chunk}

It is worth noting that the inequality in \ref{rcxrpd}(i) may fail if the module in question does not have finite complexity; indeed, for a local ring $(R,\m,k)$ with $\m^2=0$, we always have $\rpdim(k)<\infty$ (\cite[Proposition 2.5]{CA}). Similarly, the equality in  \ref{rcxrpd}(ii) may fail if the module in question does not have reducible complexity: the module in the next example has finite complexity, but it does not have finite reducing projective dimension (and hence the module does not have reducible complexity). This shows that, for a module, having finite complexity and having finite reducing projective dimension are independent conditions, in general. 

\begin{eg} Jorgensen and {\c{S}}ega \cite{JS} constructed a local Artinian ring $R$ and an $R$-module $M$ such that $\cx_R(M)<\infty=\Gdim_R(M)$ and $\Ext^i_R(M,R) = 0$ for all $i\geq 1$. It follows that $\rGdim_R(M)=\infty$ as otherwise the vanishing of $\Ext^i_R(M,R)$ forces $M$ to have finite Gorenstein dimension; see \cite[1.3]{CA}. \qed
\end{eg}

\section{Testing the Gorenstein property via reducing Gorenstein dimension}\label{sec 3}
In this section, we concern reducing Gorenstein dimension and prove the first main theorem, which has been advertised in the introduction.

\begin{chunk} \label{C1}
Let $R$ be a local ring and $0 \to A \to B \to C \to 0$ be a short exact sequence of $R$-modules. Then, the following hold: 
\begin{enumerate}[\rm(i)]
\item There exists a short exact sequence $0\to \syz_R C \to A \oplus F \to B \to 0$ with some free $R$-module $F$. 
\item There exists a short exact sequence $0\to \syz_R B \to \syz_R C \oplus G \to A \to 0$ with some free $R$-module $G$.
\item For each $i\ge 0$, there exists a short exact sequence $0 \to \syz_R^i A \to \syz_R^i B\oplus H_i \to \syz_R^i C \to 0$ with some free $R$-module $H_i$. 
\end{enumerate} 

\begin{proof} (i) and (iii) follow from \cite[Proposition 2.2]{dt}. For part (ii), we consider the following pull-back diagram obtained via the exact sequences $0\to \syz_R B \to F \to B \to 0$ by $0\to A \to B \to C \to 0$:  

$$
\xymatrix{
            &                                          & 0 \ar[d]                       & 0 \ar[d]           &   \\
0 \ar[r] & \Omega_R B \ar[r] \ar@{=}[d] & X \ar[r] \ar[d]             & A \ar[d] \ar[r] & 0 \\
0 \ar[r] & \Omega_R B \ar[r]                       & F \ar[r] \ar[d]             & B \ar[r] \ar[d] & 0 \\
            &                                          & C \ar@{=}[r] \ar[d] & C \ar[d]           &   \\
            &                                          & 0                                 & 0                     &  
}
$$

The middle column yields $X\cong \syz_R C\oplus G$ for some free $R$-module $G$. Hence, the required short exact sequence follows from the upper row. 
\end{proof}
\end{chunk}

\begin{chunk} Recall that $\mod R$ denotes the category of finitely generated $R$-modules. Let $\lX$ be a full and strict subcategory of $\mod R$. Given an integer $n\ge 1$, by $\syz^n \lX$, we denote the subcategory of all $R$-modules $M$ for which there exists an exact sequence $0\to M \to P_{n-1}\to \cdots \to P_0\to N \to 0$, where $N\in \lX$ and each $P_i$ is a finitely generated  projective (=free when $R$ is local) $R$-module. Moreover, we put $\syz^0\lX=\lX$. We always have $\syz^m(\syz^n \lX)=\syz^{m+n} \lX$.  
We also often denote $\syz^n(\mod R)$ by $\Syz_n(R)$.  

We say that $\lX$ is additive if it is closed under direct summands and finite direct sums. We say that $\lX$ is extension-closed (resp. closed under kernels of epimorphisms) if given any short exact sequence $0\to N \to L \to M \to 0$, we have  $M,N\in \lX \implies L \in \lX$ (resp. $M,L\in \lX \implies N\in \lX$).  We say $\lX$ is resolving if $\lX$ is additive, $R\in \lX$, and $\lX$ is closed under both extensions and kernels of epimorphisms.  
\end{chunk}  

\begin{chunk}  
An $R$-module $M$ is said to satisfy $(\widetilde S_n)$ (resp. $(S_n)$) if $\depth_{R_{\fp}}(M_{\fp})\ge \inf \{ n, \depth (R_{\fp}) \}$ (resp. $\depth_{R_{\fp}}(M_{\fp})\ge \inf \{ n, \dim (R_{\fp}) \}$) for all $\fp \in \Spec(R)$. We denote by $\widetilde S_n(R)$ (resp. $S_n(R)$) the collection of all $R$-modules that satisfy $(\widetilde S_n)$ (resp. $(S_n)$). It is easy to observe (by the depth lemma etc.) that $\widetilde S_n(R)$ is a resolving subcategory of $\mod R$. We note that $R$ satisfy $(S_n)$ if and only if $\widetilde S_n(R)=S_n(R)$. 
\end{chunk} 


\begin{dfn}  For an $R$-module $M$ we denote by $\Tr_R M$ the {\em (Auslander-Bridger) transpose} of $M$.
This is defined as follows.
Take a presentation $P_1\xrightarrow{f}P_0\to M\to 0$ by finitely generated projective $R$-modules $P_1,P_0$.
Dualizing this by $R$, we get an exact sequence $0\to M^\ast\to P_0^\ast\xrightarrow{f^\ast}P_1^*\to\Tr_R M\to0$, that is, $\Tr_R M$ is the cokernel of the map $f^\ast$. It is clear that $\Tr_R M$ is also finitely generated.   
The transpose of $M$ is uniquely determined up to projective summands; see \cite{AB} for basic properties. When the ring in question is clear, we simply write $\Tr$ in place of $\Tr_R$.    
\end{dfn}

\begin{chunk}
An $R$-module $M$ is said to be {\it $n$-torsionfree} if $\Ext_R^i(\Tr_R M, R) = 0$ for all $i=1,\ldots n$.
We denote by $\TF_n(R)$ the full subcategory of $\mod R$ consisting of $n$-torsionfree $R$-modules. 	
\end{chunk}

\begin{chunk}\label{DS}(\cite[Proposition 2.4]{DS})
We always have inclusions $\TF_n(R) \subseteq \Syz_n(R) \subseteq \widetilde{S}_n(R)$. Moreover, if $M_\fp$ is totally reflexive for prime ideals $\fp$ with $\depth R_\fp < n$ and satisfies $(\widetilde{S}_n)$, then $M$ is $n$-torsionfree.
\end{chunk}   

For a local ring $R$, let $\cC(R)$ denote the full subcategory of all $R$-modules $M$ such that $\depth_R(M)\geq \depth(R)$\footnote[1]{$R$-modules we consider here are called {\it deep} in \cite{DEL} and the category $\cC(R)$ is denoted by $\mathrm{Deep}(R)$ there.}. It is easily observed that $\cC(R)$ is a resolving subcategory of $\mod R$, and $\widetilde S_n(R)\subseteq \cC(R)$ for all $n\ge \depth(R)$. When $R$ is local Cohen--Macaulay, it holds that $\tilde S_n(R)=S_n(R)=\cC(R)$ for all $n\ge \depth (R)$ and this is nothing but the category of maximal Cohen-Macaulay $R$-modules. In general, for $n\ge \depth R$, the inclusion  $\widetilde S_n(R)\subseteq \cC(R)$ can be strict.  
 
We need the following preliminary result and its consequence to prove the main result of this section.  

\begin{lem} \label{prop-app} Let $\lX$ be an extension-closed subcategory of $\mod R$. If $0\to M \to K \to N \to 0$ is a short exact sequence such that $K\in \syz \lX$ and $N\in \lX$. Then it follows that $M\in \syz \lX$.  
\end{lem}

\begin{proof} By assumption, we have a short exact sequence $0\to K \to F \to K'\to 0$, where $F$ is free and $K'\in \lX$. This yields a pushout diagram as follows:
\begin{align}\notag{}
\begin{aligned}
\xymatrix{
&& 0 \ar[d] & 0 \ar[d]\\
0 \ar[r] & M \ar[r] \ar@{=}[d] \ar@{}[dr]& K \ar[r] \ar[d]  & N \ar[r] \ar[d] & 0 \\
0 \ar[r] & M \ar[r]  & F \ar[r] \ar[d] & C \ar[r] \ar[d] & 0 \\
&& K' \ar[d] \ar@{=}[r] & K' \ar[d] \\ 
&& 0 & 0
}  
\end{aligned}
\end{align}
Note, as $\lX$ is extension-closed and $N,K'\in \lX$, it follows that $C\in \lX$. Thus $0\to M \to F \to C \to 0$ implies $M\in \syz\lX$.  
\end{proof}

However the following result is recorded in \cite[3.1(2)]{restf}, we give another proof based on Lemma \ref{prop-app}.

\begin{cor}\label{syzadd} Let $\lX$ be a resolving subcategory of $\mod R$. Then $\syz \lX$ is additive (i.e., it is closed under finite direct sums, and direct summands). 
\end{cor}

\begin{proof} Since $\lX$ is closed under finite direct sums, hence so is $\syz\lX$. We only need to show that if $M,N\in \mod R$ such that $M\oplus N\in \syz \lX$, then $M\in \syz\lX$. 
Indeed, if $M\oplus N \in \syz \lX$, then $M\oplus N\in \lX$ as $\lX$ is resolving, and therefore $M,N\in \lX$. From the split exact sequence $0\to M \to M \oplus N \to N \to 0$, in view of Lemma \ref{prop-app}, we get $M\in \syz \lX$.
\end{proof}    

We denote by $\cG(R)$ the category of totally reflexive modules (\cite{Gdimbook}). Note that $\syz \cG(R)= \cG(R)$ (see \cite[Lemma 1.1.10 and Theorem 4.1.4]{Gdimbook}). Moreover, by the Auslander-Bridger formula \cite[Theorem 1.4.8]{Gdimbook}, we have the inclusion $\cG(R) \subseteq \cC(R)$. Conversely, if $M$ has finite Gorenstein dimension and belongs to $\cC(R)$, then $M$ is totally reflexive by Auslander-Bridger formula. 

Now we state and prove the first main result of this section which will be frequently used in the proofs of the rest of the results of this section.  

\begin{thm} \label{cor-app} Let $R$ be a local ring and let $\lX$ be a resolving subcategory of $\mod R$ such that $\cG(R)\subseteq \lX\subseteq \cC(R)$. If $M$ is an $R$-module with $\rGdim_R(M)<\infty$, then $M\in \lX$ if and only if $M\in \Omega\lX$.  
\end{thm}

\begin{proof}  Note it is enough to assume $M\in \lX$ and show that $M\in \Omega\lX$. Hence we assume $M\in \lX$. First we note that $\cG(R)=\syz \cG(R)\subseteq \syz \lX$. 
It follows, as $\rGdim_R(M)<\infty$, there are short exact sequences of $R$-modules
\begin{equation}\tag{\ref{cor-app}.1}
E_i=(0 \to K_{i-1}^{{\oplus a_i}} \to K_{i} \to \Omega^{n_i}_RK_{i-1}^{{\oplus b_i}} \to 0)
\end{equation}
for each $i=1, \ldots r$, where $a_1, \dots, a_r,b_1, \dots, b_r$ are positive integers, $n_1, \dots, n_r$ are non-negative integers, $K_0=M$, and $\Gdim_R(K_r)<\infty$. As $\lX$ is resolving, so is each $K_i\in \lX$. In particular, $K_r\in \lX\subseteq \cC(R)$ and thus $\Gdim(K_r)<\infty$ implies that $K_r$ is totally reflexive. 
Consequently, $K_r\in \cG(R) \subseteq \syz \lX$.
Since $\syz^{n_r}_RK_{r-1}^{\oplus b_r}\in \lX$ and $K_r \in \syz \lX$, $K_{r-1}^{\oplus a_r}\in \syz\lX$ by Lemma \ref{prop-app} applied to $E_r$. Then, $K_{r-1}\in \syz\lX$ by Corollary \ref{syzadd}. Similar argument applied to $E_{r-1}$ shows $K_{r-2}\in \syz\lX$ and so on. Continuing this way, we get $M=K_0\in \syz \lX$. 
\end{proof}  

The following is an interesting consequence of Theorem \ref{cor-app}. 

\begin{cor}\label{res2}
Let $R$ be a local ring and let $\lX$ be a resolving subcategory of $\mod R$ with $\cG(R) \subseteq  \lX$.  
Then the following conditions are equivalent:
\begin{enumerate}[\rm(i)]
\item $\Gdim_R(M)< \infty$ for each $M \in \cX$.
\item $\rGdim_R(M)<\infty$ for each $M \in \cX$.
\item There exists $n\ge \depth R$ such that $\rGdim_R(M)<\infty$ for each $M \in \lX \cap \widetilde S_n(R)$. 
\end{enumerate} 
In particular, for a resolving subcategory $\lX$ of $\cC(R)$ with $\cG(R) \subseteq \lX$, $\cX = \cG(R)$ if and only if $\rGdim_R(M)<\infty$ for each $M \in \cX$.
\end{cor}

\begin{proof}Set $t = \depth R$.
Assume (i), and let $M\in \lX$. Then $\syz^t M$ is totally reflexive by assumption of (i). Hence, $\Gdim_R M<\infty$, and so $\rGdim_R M<\infty$. This shows the implication (i)$\implies$ (ii). 
On the other hand, the implication (ii)$\implies$ (iii) is obvious.  

Now, assume that (iii) holds. The assumptions on $\lX$ imply that $\cY:=\lX \cap \widetilde{S}_n(R)$ is a resolving subcategory of $\cC(R)$ and $\cY$ contains $\cG(R)$. Then by Theorem \ref{cor-app}, $\cY=\syz\cY$ and therefore $\cY = \syz^t \cY = \cG(R)$ holds.
Therefore, the remained implication (iii)$\implies$ (i) follows.  
\end{proof}  

\begin{rem}\label{eg}Examples of resolving subcategories $\lX$ that satisfy the hypothesis of Theorem \ref{cor-app} include $\lX=\widetilde S_n(R)$ for $n\ge \depth R$, and also $\lX=\cC(R)$.  
\end{rem}   

Now we prove the main theorem of this section, characterizing Gorenstein local rings via reducing G-dimension. In the following, for given subcategories $\lX,\Y$ of $\mod R$, denote by $\lX *\Y$ the collection of all modules $L$ which fits into an exact sequence $0\to X \to L \to Y\to 0$, where $X\in \lX, Y\in \Y$.   

\begin{thm} \label{Gdimcor} Let $(R,\m,k)$ be a local ring. Then the following conditions are equivalent:
\begin{enumerate}[\rm(i)]
\item $R$ is Gorenstein.
\item $\rGdim_R(M)<\infty$ for each $R$-module $M$.  
\item There exists a resolving subcategory $\lX$ of $\mod R$ containing $\syz^n_Rk$ for some $n \ge 0$ such that $\rGdim_R(M)<\infty$ for each $M\in \lX$.
\item There exists a category $\lX \subseteq \mod R$ closed under extensions and containing $R$ and $\syz^n_R k \in \lX$ for some $n\ge \depth R$, such that $\rGdim_R(M)<\infty$ for each $M\in \lX$.
\item  There exists $n\ge \depth R$ such that $\rGdim M<\infty$ for each $M\in R*\syz^n_R k$. 
\item There exists $M\in  \cC(R)$ such that $\id_R (M)<\infty$ and $\rGdim_R(M)<\infty$.
\item $\rGdim M < \infty$ for each $R$-module $M$ with $\Omega^m_R M \cong \Omega^n_R k $ for some positive integers $m$ and $n$. 
\item There exists an integer $n\ge \depth R$ such that  $\rGdim_R(\Tr_R \syz^n_R k)<\infty$.
\end{enumerate}  
\end{thm}  

\begin{proof} (i)$\implies$ (ii)$\implies$ (iii)$\implies$ (iv)$\implies$ (v) is obvious.  

Now to show (v)$\implies$(i):  Let $t=\depth R$. We notice that $R*\syz^n_Rk \subseteq \widetilde S_n(R)\subseteq \widetilde S_t(R)$. Take an $R$-module $M\in R*\syz^n_Rk$. Then $M\in \tilde S_t(R)$ and $\rGdim_R (M)<\infty$ by the assumption of (vi). Now,  Theorem \ref{cor-app} and Remark \ref{eg} imply that $M \in \syz \widetilde S_t(R)$. Then $M\cong\Omega_R N \oplus F$ for some $N \in \widetilde{S}_t(R)$ and some free $R$-module $F$. For any prime ideal $\fp$ with $\depth R_\fp < t$, $M_\fp$ is free (as $\fp\ne \m$ and $M$ is locally free on punctured spectrum), and hence $\pd_{R_{\fp}} N_{\fp} \le 1$. As $N$ satisfies $\widetilde{S_t}$, it follows that $\depth N_{\fp}\ge \inf\{t,\depth R_{\fp}\}=\depth R_{\fp}$. Due to the Auslander-Buchsbaum formula $\pd_{R_{\fp}} N_{\fp}=\depth R_{\fp}-\depth N_{\fp}\le 0$, $N_{\fp}$ is free. Now we have shown that $N_{\fp}$ is $R_{\fp}$-free for all prime ideal $\fp$ satisfying $\depth R_{\fp}<t$. Therefore, it follows from \ref{DS} that $N \in \Syz_t(R)$, and so $M\cong \syz_R N\oplus F\in \Syz_{t+1}(R)$. Thus we conclude that $R*\syz^n_Rk \subseteq \Syz_{t+1}(R)$ and it implies that $R$ is Gorenstein by \cite[5.4]{restf}.  

This shows that (i) through (v) are all equivalent.  

To show (i)$\iff$(vi), we only need to show (vi)$\implies$(i): So, assume $\rGdim_R(M)<\infty$ for some $M\in \cC(R)$, where $\id_R(M)<\infty$. Then, by Theorem \ref{cor-app} and Remark \ref{eg}, we see $M \in \Omega \cC(R)$. Therefore, there is a short exact sequence of $R$-modules $0 \to M \to F \to N \to 0$, where $F$ is free and $N\in \cC(R)$. Since $\id_R(M)<\infty$ and $N\in \cC(R)$, it follows $\Ext^1_R(N, M)=0$. Hence the short exact sequence  $0 \to M \to F \to N \to 0$ splits so that $M$ is free and hence $R$ is Gorenstein.

To show (vii)$\iff $(i), enough to show (vii)$\implies$(i): Let $t = \depth R$.
For an $R$-module $M \in R*\syz^tk$, there is a short exact sequence $0 \to R \to M \to \Omega^t_R k \to 0$.	 This short exact sequence implies, by \ref{C1}(i), that $F\oplus \Omega_R M \cong \Omega^{t+1}_Rk$ for some free module $F$, and hence $\Omega^2_R M \cong \Omega^{t+2}_Rk$. Therefore, the assumption forces $\rGdim M < \infty$ and hence $R$ is Gorenstein by (v)$\implies $(i).  
 
Finally, to show (viii)$\iff $(i), enough to show (viii)$\implies$(i): Let $t=\depth R$, and we assume there exists an integer $n\ge t$ such that $\rGdim_R(\Tr_R \syz^n_R k)<\infty$. Put $M:=\Tr_R \syz^n_R k$. We first show that $M$ can be embedded in a module of finite projective dimension. This is of course true if $\Gdim_R M<\infty$ by \cite[Lemma 2.17]{cfh}. If $\Gdim_R(M)=\infty$, then by Definition \ref{rdim} there exist integers $r,a_i,b_i\geq 1$, $n_i\geq 0$, and short exact sequences of $R$-modules of the form 
\begin{equation}
0 \to K_{i-1}^{{\oplus a_i}} \to K_{i} \to \Omega^{n_i}_RK_{i-1}^{{\oplus b_i}} \to 0
\end{equation}
for each $i=1, \ldots r$, where $K_0=M$ and $\Gdim(K_r)<\infty$.  Since $K_r$ embeds in a module of finite projective dimension, hence $K_{r-1}^{\oplus a_r}$ embeds in a module of finite projective dimension, so $K_{r-1}$ embeds in a module of finite projective dimension. Similarly, $K_{r-2}$ embeds in a module of finite projective dimension, and continuing this way, $M=K_0$ embeds in a module of finite projective dimension. So in any case, we now see that $M$ embeds into a module of finite projective dimension, say $H$. Now, since $n\ge t$, so $\syz^n_R k$ satisfies $(\tilde S_t)$, so by \cite[Proposition 2.4]{DS}, we get  $\Ext^i_R(M,R)=0$ for all $1\le i\le t$. Since $\pd_R H\le t$, so by \cite[Lemma 2.2]{tyy}, we get $M$ is torsion-less i.e. $\Ext^1_R(\Tr_R M,R)=0$. Since $\Tr \Tr \syz^n_R k$ is stably isomorphic with $\syz^n_R k$, we get $\Ext^1_R(\syz^n_R k, R)=0$ i.e., $\Ext^{n+1}_R(k, R)=0$. By \cite[II. Theorem 2]{rob} we get $\id_R R<\infty$ i.e., $R$ is Gorenstein. 
\end{proof}    

\begin{rem}
However the equivalence between (i) and (vi) has been proved	 in \cite[Corollary 3.3(iii)]{CA}, our argument is more simpler than their argument.   
\end{rem}

The following characterization of local complete intersection rings shows that assuming the subcategory of all modules of finite $\rGdim$ contains a big enough subcategory which is closed under extensions, imposes strong conditions on the ring.  

\begin{prop} \label{p1ci} Let $(R,\m,k)$ be a $d$-dimensional local ring such that $\widehat R=S/(x_1,\ldots,x_n)S$, where $(S,\mathfrak{n},k)$ is local Cohen--Macaulay ring of minimal multiplicity, and $x_1,\ldots,x_n \in \mathfrak n$ is an $S$-regular sequence. Then the following are equivalent:  

\begin{enumerate}[\rm(1)]
    \item $R$ is a complete intersection. 
    
    \item $\rpdim_R M<\infty$ for all $M\in \mod R$. 
    
    \item There exists a subcategory $\lX \subseteq \mod R$ which is closed under extensions such that $$\{M\in \TF_{d+1}(R)|M \text{ is locally free on the punctured spectrum, and } \rpdim_R M<\infty \}\subseteq \lX$$ and $$\lX \subseteq \{M\in \mod R| \rGdim_R M<\infty\}.$$  
\end{enumerate}   
\end{prop}   

\begin{proof} $(1)\implies (2)\implies (3)$ is straightforward.
Only need to prove $(3)\implies (1)$: Assume the existence of a subcategory $\lX$ as in (3). 
By hypothesis, $R\in \lX$. Since $\syz^{d+1}_R k$ is locally free on the punctured spectrum, and is $(d+1)$-torsionfree, and moreover $\rpdim_R \syz^{d+1}_R k < \infty$ by \cite[Theorem 1.2]{redring}, so $\syz^{d+1}_R k \in \lX$. Since $\lX$ is closed under extensions, so $R*\syz^{d+1}_R k\subseteq \lX$. So, $\rGdim_R M<\infty$ for all $M\in R*\syz^{d+1}_R k$ by hypothesis on $\lX$. Then, $R$ is Gorenstein by Theorem \ref{Gdimcor}. So $S$ is Gorenstein. but $S$ has minimal multiplicity, so $S$ is a hypersurface. Hence, $R$ is a complete intersection.  
\end{proof} 

We have another consequence of Theorem \ref{cor-app} regarding Ulrich modules (see \cite{buh},\cite{umm}), which extends and recovers \cite[Proposition 2.5((vi)$\implies$ (i))]{CA}.  

\begin{cor} Let $(R,\m,k)$ be a local Cohen--Macaulay ring of minimal multiplicity. Let $M$ be a maximal Cohen--Macaulay $R$-module such that $\rGdim_R M<\infty$. Then, $M\cong N\oplus F$ for some Ulrich module $N$ and free $R$-module $F$. In particular, if $\m^2=0$ and $M$ is an $R$-module such that $\rGdim_R M<\infty$, then $M\cong k^{\oplus a}\oplus F$ for some integer $a\ge 0$ and a free $R$-module $F$. 
\end{cor}

\begin{proof}Taking $\lX=\cC(R)$ (the full subcategory of all maximal Cohen--Macaulay modules) in Theorem \ref{cor-app}, it follows that $M\in \syz \cC(R)$. Hence we can write $M\cong F \oplus N$ for some $N\in \syz \cC(R)$ such that $R$ is not a direct summmand of $N$. Since $R$ has minimal multiplicity, it follows from \cite[Proposition 1.6]{umm} that $N$ is an Ulrich module. 
Finally, if $\m^2=0$, then $R$ is a Cohen--Macaulay ring of minimal multiplicity, and every $R$-module is maximal Cohen--Macaulay. Since $R$ is Artinian, all Ulrich modules are $k$-vector spaces (see \cite[Proposition (1.2)]{buh}). Hence, the claim follows. 
\end{proof}

\section{Testing the complete intersection property via reducing projective dimension}\label{sec 4}

This section is devoted to a proof of the following theorem which contains Theorem\ref{main} and gives a partial affirmative answer to Question \ref{qci}.  

\begin{thm}\label{mainci}  Let $(R,\m,k)$ be a local ring and let $0 \le c \le 2$ be an integer. Then the following are equivalent:  
\begin{enumerate}[\rm(i)]
\item $R$ is a complete intersection of codimension at most $c$.
\item $\rpdim_R(M) \le c$ for each  $R$-module $M$. 
\item $R$ is Gorenstein and $\rpdim_R(k) \le c$.  
\item There exists an integer $n\ge \depth R$ such that  $\rpdim_R(\Tr_R \syz^n_R k)\le c$.       
\end{enumerate}  
\end{thm}


Since the case $c=0$ is trivial, we proceed and prove the cases $c=1$ and $c=2$ separately.  

On both cases, the key ingredient of the proof is the following theorem; in order not to interrupt the flow of the arguments, we defer the proof until the end of the paper.

\begin{thm}\label{corcx} Let $R$ be a local ring and let
$$
0 \to M^{\oplus a} \to K \to \Omega^n_R M^{\oplus b} \to 0 
$$
be a short exact sequence of $R$-modules such that $\cx_R(K) =c$ for some integer $c\geq 0$. Then the following hold:
\begin{enumerate}[\rm(i)]
\item
If $a < b$, then $\cx_R(M) =\cx_R(K)$.
\item	
If $a = b$, then $\cx_R(M)\le \cx_R(K) +1$.
\item
If $a > b$, then $\cx_R(M)= \cx_R(K)$ or $\cx_R(M)=\infty$.
\end{enumerate}
\end{thm}

However the following result has been mentioned in \ref{rcxrpd}, we are now ready to prove it.

\begin{cor} \label{corcxrpd} Let $R$ be a local ring and let $M$ be an $R$-module.
\begin{enumerate}[\rm(i)]
\item If $\cx_R(M) < \infty$, then it follows that $\cx_R(M) \le \rpdim_R(M)$.
\item If $M$ has reducible complexity (e.g., if $\CIdim_R(M)<\infty$), then $\cx_R(M) = \rpdim_R(M)$.
\end{enumerate}
\end{cor}

\begin{proof} The claim in part (ii) is an immediate consequence of part (i): we know, if $M$ has  reducible complexity, then it follows that $\rpdim_R(M) \leq \cx_R(M)$; see \cite[Theorem 3.6]{at}. Hence it is enough to prove the claim in part (i).

Set $r=\rpdim_R(M)$. Note that we may assume $r<\infty$. We proceed by induction on $r$. The claim follows by definition in case $r=0$. We shall assume $r\geq 1$. Then there is a short exact sequence of $R$-modules
\begin{equation}
0 \to M^{\oplus a} \to K \to \Omega^n_R M^{\oplus b} \to 0,
\end{equation}
where $a, b \ge 1$, $n \ge 0$, and $\rpdim_R(K) = r - 1$. 
We note that $\cx_R(K)<\infty$ because $\cx_R(M)<\infty$.
By the induction hypothesis, we have that $\cx_R(K) \le \rpdim_R(K) = r- 1$. 

If $a \le b$, then $\cx_R (M) \le \cx_R (K) + 1 \le \rpdim_R (M)$ by Theorem \ref{corcx}(i)(ii). On the other hand, if $a > b$, then since $\cx M<\infty$, we have that $\cx_R (M) =\cx_R (K) <r=\rpdim_R (M)$ by Theorem \ref{corcx}(iii).
\end{proof}  

Next, we show a general result that a totally reflexive module and its $R$-dual have the same reducing projective dimension. This will be used for the proof of Theorem \ref{mainci}. 

\begin{lem} \label{lem57} Let $R$ be a local ring and $M$ be a totally reflexive $R$-module.
If there is a short exact sequence of $R$-modules
\begin{equation} \tag{\ref{lem57}.1}
0 \to M^{\oplus a} \to K \to \syz^n_R M^{\oplus b} \to 0
\end{equation}
such that $a, b\geq 1, n\geq 0$ are integers, then 
there is a short exact sequence
\begin{equation} \tag{\ref{lem57}.2}
0 \to (M^*)^{\oplus b} \to \syz^n_R(K^*)\oplus F \to \syz^n_R (M^*)^{\oplus a} \to 0,
\end{equation}
with some free $R$-module $F$. 
\end{lem}

\begin{proof} Applying $(-)^*$ to the sequence (\ref{lem57}.1), we get a short exact sequence
\begin{equation} \notag{} 
0 \to (\syz^n_R M^{\oplus b})^* \to K^* \to (M^{\oplus a})^* \to 0.
\end{equation}
Taking $n$-th syzygies, we obtain a short exact sequence
\[
0 \to \syz^n_R((\syz^n_R M)^*)^{\oplus b} \to \syz^n_R(K^*)\oplus F' \to \syz^n_R (M^*)^{\oplus a} \to 0
\]
for some free $R$-module $F'$. Since $M$ is totally reflexive, one has $M^*\cong \syz^n_R((\syz^n_R M)^*)\oplus G$ for some free $R$-module $G$. Therefore, putting $F:=F'\oplus G$, the above exact sequence gives 
\[
0 \to (M^*)^{\oplus b} \to \syz^n_R(K^*)\oplus F \to \syz^n_R (M^*)^{\oplus a} \to 0.
\]  
\end{proof} 

\begin{prop}\label{57} Let $R$ be a local ring and let $M$ be a totally reflexive $R$-module.
Then it follows $\rpdim M=\rpdim M^*$.
\end{prop}

\begin{proof} 
It suffices to prove the inequality $\rpdim M^{\ast}\leq \rpdim M$ for any totally reflexive $R$-module.
Indeed, since $M^*$ is also a totally reflexive and $M \cong M^{**}$, we get the inequality $\rpdim M=\rpdim M^{\ast \ast}\leq \rpdim M^{\ast}$. 
Consequently, we obtain $\rpdim M=\rpdim M^*$.

To prove $\rpdim M^{\ast}\leq \rpdim M$, we assume $r:=\rpdim M<\infty$ and proceed by induction on $r$.

If $r=0$ (i.e., $\pd M<\infty$), then $M$ is free and hence so is $M^*$. This means that $\rpdim M^*=0$. Next, assume $r > 0$ and choose an exact sequence as (\ref{lem57}.1) in Lemma \ref{lem57} such that $\rpdim K=r-1$. Since $K$ is totally reflexive, the induction hypothesis yields $\rpdim K^*\leq \rpdim K$. Using \cite[Proposition 3.8]{redring}, we get $\rpdim (\syz^n_R(K^*)\oplus F)=\rpdim \syz^n_R(K^*)\le \rpdim K^*\le \rpdim K = r-1$. Thus the sequence (\ref{lem57}.2) gives us the inequality $\rpdim M^* \le r$.
\end{proof}

Now we proceed with the proof of $c=1$ case of Theorem \ref{mainci}.   

\subsection*{Proof of Theorem \ref{mainci} when $c=1$} 

\begin{chunk} \label{c2} Let $R$ be a complete local ring, $M$ be an $R$-module, and let $\ind(M)$ denote the number of non-free summands in a direct sum decomposition of $M$ by indecomposable modules(such a decomposition exists and is unique as $R$ is complete, see \cite[Corollary 1.10]{lw}). Clearly, $\ind(M)=0$ if and only if $M$ is free. 
Now assume $M$ is totally reflexive. Then, any direct summand $X$ of $M$ is totally reflexive, and moreover $X$ is free if and only if $X^*$ is free, if and only if $\syz_R X$ is free. So, we have that $\ind(M) \leq \ind(\Omega_R M)$ and $\ind(M)\le \ind(M^*)$. Similarly, $\ind(M^*)\le \ind(M^{**})=\ind(M)$. So we get $\ind(M^*)=\ind(M)$. Now remembering $M^*\cong F\oplus \syz_R((\syz_R M)^*)$ for some free $R$-module $F$, we get $\ind(M)\le\ind(\syz_R M)=\ind((\syz_R M)^*)\le \ind(\syz_R((\syz_R M)^*))=\ind(M^*)=\ind(M)$. Thus, $\ind(M)=\ind(\syz_R M)$. 
\end{chunk} 

\begin{chunk}  \label{c3} Let $R$ be a local ring and let $M$ be a totally reflexive $R$-module. Assume there exists a short exact sequence of $R$-modules $0 \to M^{\oplus a} \to F \to \Omega^n_R M^{\oplus b} \to 0$, where $a, b \ge 1$, $n \ge 0$ are integers and $F$ is free. Then it follows that $\cx_R(M)\leq 1$. 


Since completion commutes with syzygy and direct sum, completion of totally reflexive modules are totally reflexive, and complexity does not change under completion, so we may pass to completion and assume without loss of generality that $R$ is complete. Hence, we can talk about $\ind(-)$.  

Let $r=\ind(M)$. Since $ M^{\oplus a} \cong \Omega^{n+1}_R M^{\oplus b} \oplus G$ for some free $R$-module $G$, so we get $ra = \ind(M^{\oplus a})=\ind( \Omega^{n+1}_R M^{\oplus b}\oplus G)=\ind(\syz^{n+1}_R M^{\oplus b})=\ind( M^{\oplus b})=r b$. This yields $r=0$ or $a=b$. 
If $r=0$, then $M$ is a free $R$-module and hence $\cx_R(M) = 0$ by definition.
On the other hand, if $a=b$, then $M^{\oplus a}\cong F\oplus \Omega^{n+1}_R M^{\oplus a}$, and this implies $\syz_R M^{\oplus a}\cong \syz^{n+2}_R M^{\oplus a}$. From this isomorphism, we conclude that $M$ has bounded Betti numbers and hence $\cx_R(M) \le 1$.\qed  
\end{chunk}  

If $R$ is a local ring and $M$ is an $R$-module, then it is clear by definition that $\cx_R(M)=0$ if and only if $\rpdim_R(M)=0$. In the following we make a similar observation, connecting modules of bounded Betti numbers with those having reducing projective dimension $1$. For that we first prepare a lemma:   

\begin{lem}\label{period} Let $M$ be an $R$-module such that $M^{\oplus a}$ has a periodic resolution for some $a\ge 1$. Then $M$ also has a periodic resolution.   
\end{lem} 

\begin{proof} Without loss of generality we may assume $R$ is complete. Assume there exists an integer $n\ge 1$ such that $M^{\oplus a} \cong \syz^n_R (M^{\oplus a}).$ Write $\syz^n_R M\cong \oplus_{i=1}^r K_i^{\oplus b_i}$ where $K_i$'s are mutually non-isomorphic indecomposable modules. Then we have $\syz^n_R M^{\oplus a} \cong \oplus_{i=1}^r K_i^{\oplus ab_i}$. Write $M\cong \oplus_{i=1}^s N_i^{\oplus c_i}$ where $N_i$'s are mutually non-isomorphic indecomposable modules. Then, $\oplus_{i=1}^s N_i^{\oplus ac_i}\cong M^{\oplus a}\cong \syz^n_R M^{\oplus a}\cong \oplus_{i=1}^r K_i^{\oplus ab_i}$ and uniqueness of indecomposable decomposition implies $s=r$, and after a permutation, we can write $N_i\cong K_i$ and $ac_i=ab_i$. Hence, $c_i=b_i$ and $M\cong \oplus_{i=1}^r K_i^{\oplus b_i}\cong \syz^n_R M$.  
\end{proof}

\begin{prop} \label{p1} Let $R$ be a local ring and let $M$ be an $R$-module. 
\begin{enumerate}[\rm(i)]
\item If $\Gdim_R(M)<\infty$ and $\rpdim_R(M)\le 1$, then $\Omega^{n} M$ has a periodic resolution for some $n\geq 0$.
\item If $R$ is Gorenstein and $\rpdim_R(k)\leq 1$, then $R$ is a hypersurface.

\item If $M$ has finite complexity and $\rpdim_R (M)\le 1$, then $\syz^n M$ has periodic resolution for some $n>0$.  


\item If there exists an integer $n\ge \depth R$ such that  $\rpdim_R(\Tr_R \syz^n_R k)\le 1$, then $R$ is a hypersurface.
\end{enumerate}   
\end{prop}

\begin{proof} To prove part (i), we assume $\Gdim_R(M)<\infty$ and $\rpdim_R(M)\le 1$. If $\rpdim_R(M)=0$, then $\pd_R(M)<\infty$ in which case there is nothing to prove. Therefore, we can assume $\rpdim_R(M)=1$. Set $d=\depth(R)$ and $N=\Omega_R^d M$. Then $N$ is totally reflexive and $\rpdim_R(N)=1$ ($\rpdim N\ne 0$ since $N$ has infinite projective dimension). Let $0 \to N^{\oplus a} \to F \to \Omega^n_R N^{\oplus b} \to 0$ be a reducing $\pd$-sequence. It follows from the argument of \ref{c3} that $a=b$, hence we see that $N^{\oplus a}$ has periodic resolution and so does $N=\syz^d M$ by Lemma \ref{period}.   

Note, if $R$ is Gorenstein and $\rpdim_R(k)\leq 1$, then part (i) implies that $\cx_R(k)\leq 1$, i.e., $R$ is a hypersurface. Therefore part (ii) holds. 

For part (iii), again we may assume $\rpdim_R(M)=1$. Consider a reducing sequence  \begin{align}
0 \to M^{\oplus a_0} \to K_{1} \to \Omega^{n_0}_R M^{\oplus b_0} \to 0  
\end{align}
such that  $K_1$ has finite projective dimension, $a_0,b_0$ are positive integers. Since $M$ has infinite projective dimension, so $\cx K_1=0<\cx M<\infty$ implies $a_0=b_0$ by  Theorem \ref{corcx}(1) and (3).  So we have the short exact sequence $0 \to M^{\oplus a_0} \to K_{1} \to \Omega^{n_0}_R M^{\oplus a_0} \to 0$. Since $\syz^n_R K_1$ is free for all $n\gg 0$, we get $\syz^n_R M^{\oplus a_0} \oplus F \cong \syz^{n+n_0+1}_R M^{\oplus a_0}$ and by taking more syzygies, we get $\syz^n_R M^{\oplus a_0}  \cong \syz^{n+n_0+1}_R M^{\oplus a_0}\cong \syz^{n_0+1}_R(\syz^n_R M^{\oplus a_0})$ for all $n\gg 0$. Thus, $\syz^n_R M$ has periodic resolution by Lemma \ref{period}.  


(iv) The hypothesis along with Theorem\ref{Gdimcor} implies that $R$ is Gorenstein. Let $d=\dim R$. There exists a free module $F$ such that $\syz^2\Tr_R \syz^n_R k \oplus F\cong  (\syz^{n}_R k)^*$. By \cite[Corollary 3.13]{redring}, we have $\rpdim_R(\syz^2\Tr_R \syz^n_R k)\le 1$. By \cite[Proposition 3.8]{redring}, $\rpdim_R(\syz^2\Tr_R \syz^n_R k  \oplus F)=\rpdim_R((\syz^{n}_R k)^*)\le 2$. Since $R$ is Gorenstein and $n\ge \dim R$, $(\syz^{n}_R k)^*$ and $\syz^n_R k$ are totally reflexive. Hence we get $\rpdim_R(\syz^n_R k)=\rpdim_R((\syz^n_Rk)^*)\le 1$ by Proposition \ref{57}. If $\rpdim_R(\syz^n_R k)=0$, then by definition $\pd_R(\syz^n_R k)<\infty$, $R$ is regular. Otherwise $\rpdim_R(\syz^n_R k)=1$, and then $R$ is a hypersurface by part (i).
\end{proof}


\begin{proof}[Proof of  $c=1$ case of Theorem \ref{mainci}]  Note that (ii) $\Longrightarrow$ (iii) follows from Theorem \ref{Gdimcor} and (iii)  $\Longrightarrow$ (i) is due to Proposition \ref{p1}(ii). Since (i) $\Longrightarrow$ (ii) is clear from Corollary \ref{corcxrpd}, so this proves $\rm{(i)}\iff\rm{(ii)}\iff \rm{(iii)}$. Finally, $\rm{(i)}\iff\rm{(iv)}$ also follows by Proposition \ref{p1}(iv) and Corollary \ref{corcxrpd}.
\end{proof}  

\subsection*{Proof of Theorem \ref{mainci} when $c=2$}

\begin{prop} \label{c1} Let $R$ be a local ring and let $M$ be a totally reflexive $R$-module. Assume $\rpdim_R(M)\leq 2$.
Then it follows that $\cx_R(M)\leq 2$ or $\cx_R(M^*) \leq 2$.
\end{prop}  

\begin{proof} By assumption, we have an exact sequence \begin{equation}
0 \to M^{\oplus a} \to K \to \Omega^n_R M^{\oplus b} \to 0,
\end{equation}
where $a, b \ge 1$, $n \ge 0$, and $\rpdim_R(K) \le 1$. Since $K$ is totally reflexive, we get $\cx_R(K)\le 1$ by Proposition \ref{p1}. If $a\le b$, then $\cx_R(M)\le \cx_R(K)+1\le 2$ by Theorem \ref{corcx}(i)(ii). Next we assume $a>b$. By Lemma \ref{lem57} and Proposition \ref{57} we have an exact sequence \[
0 \to (M^*)^{\oplus b} \to \syz^n_R(K^*)\oplus F \to \syz^n_R (M^*)^{\oplus a} \to 0,
\] where $\rpdim_R K^*=\rpdim_R K\le1$. Then \cite[Proposition 3.8]{redring} yields $\rpdim_R (\syz^n_R(K^*)\oplus F) \le 1$. Since $\syz^n_R(K^*)\oplus F$ is totally reflexive, we obtain $\cx_R (\syz^n_R(K^*)\oplus F)\le 1$ by Proposition \ref{p1}. Thus $a>b$ now implies $\cx_R(M^*)\le \cx_R (\syz^n_R(K^*)\oplus F)\le 1$ by Theorem \ref{corcx}(iii). 
\end{proof}  



\begin{prop} \label{l4} Let $(R,\m,k)$ be a local Cohen--Macaulay ring of dimension $d$.
Assume that there exists a non-zero maximal Cohen--Macaulay module $C$ of finite injective dimension such that one of the following holds
\begin{enumerate}[\rm(1)]
    \item  $\cx_R\big(\Hom_R(\syz^d_Rk,C)\big)<\infty$
    
    \item $\cx_R(C)<\infty$ (for e.g., $R$ is Gorenstein) and $\cx_R\big(\Hom_R(\syz^n_Rk,C)\big)<\infty$ for some integer $n \ge d$.
\end{enumerate}

Then, $R$ is a complete intersection. Moreover, in case \rm{(1)}, it holds that $\codim R=\cx_R (\Hom_R(\syz^d_Rk,C))$, and similarly in case \rm{(2)}, it holds that $\codim R=\cx_R (\Hom_R(\syz^n_Rk,C))$. 
\end{prop}   

\begin{proof} (1) We first prove that $R$ is a complete intersection by induction on $d=\dim R$. If $d=0$, then $\Hom_R(k,C)$ is a non-zero $k$-vector space. Therefore the hypothesis implies $\cx_R (k)<\infty$, and which implies $R$ is a complete intersection. Now assume $d>0$. Take an $R$-regular (and hence $C$-regular) element  $x\in\fm\setminus \fm^2$.
Then $(\syz^d_R k)/x(\syz^d_R k) \cong \syz^d_{R/(x)} k \oplus \syz^{d-1}_{R/(x)} k$ by \cite[Corollary 5.3]{summ}. Since $\Ext^{>0}_R(\syz^d_Rk,C)=0$, we have 
\begin{align*}
\Hom_R(\syz^d_R k, C)\otimes_R R/(x) &\cong \Hom_{R/xR}\left((\syz^d_R k)/x(\syz^d_R k), C/xC\right) \\
&\cong \Hom_{R/xR}\left(\syz^d_{R/(x)} k, C/xC\right)\oplus \Hom_{R/xR}\left(\syz^{d-1}_{R/(x)} k, C/xC\right)  	
\end{align*}
by \cite[3.3.3(a)]{bh}.
Since $x$ is $C$-regular, it is also $\Hom_R(\syz^d_R k,C)$-regular. From this, we have:

\begin{align*}\cx_{R/xR}\left(\Hom_{R/xR}\left(\syz^d_{R/(x)} k, C/xC\right)\oplus \Hom_{R/xR}\left(\syz^{d-1}_{R/(x)} k, C/xC\right) \right)&=\cx_{R/xR}\left(\Hom_R(\syz^d_R k,C)\otimes_R R/(x)\right)\\&=\cx_R\left(\Hom_R(\syz^d_R k,C)\right)<\infty. 
\end{align*}   
Hence by \cite[4.2.4(3)]{av}, we get 
$$
\max \left\{\cx_{R/xR}\left(\Hom_{R/xR}\left(\syz^d_{R/(x)} k, C/xC\right)\right), \cx_{R/xR}\left(\Hom_{R/xR}\left(\syz^{d-1}_{R/(x)} k, C/xC\right)\right)\right\}<\infty.
$$ 
In particular, $\cx_{R/xR} \left(\Hom_{R/xR}\left(\syz^{d-1}_{R/(x)} k, C/xC\right)\right)<\infty$. Since $C/xC$ is maximal Cohen--Macaulay $R/(x)$-module of finite injective dimension, the induction hypothesis yields that $R/xR$ is a complete intersection and hence so is $R$.

Now since we have proved $R$ is a complete intersection, and since $\Ext^{>0}_R(\syz^d_R k, C)=0$, we get by \cite[Thm. II(1) and 1.5(3)]{ABu} that $\cx_R(\Hom_R(\syz^d_R k, C))=\text{px}_R (\Hom_R(\syz^d_R k, C))=\cx_R(\syz^d_R k)+\text{px}_R (C)\ge \cx_R(k)$. Thus, $\cx_R(\Hom_R(\syz^d_R k, C))=\cx_R(k)=\codim R$.  

(2) The $n=d$ case is already covered by part (1). Now assume there exists an integer $n>d$ such that $\cx_R (\Hom_R(\syz^n_Rk,C))<\infty$. Applying $\Hom_R(-,C)$ to the exact sequence 
$$
0\to \syz^n_R k\to R^{\oplus b_{n-1}} \to \cdots \to R^{\oplus b_d}\to \syz^d_R k \to 0
$$ 
and remembering $\Ext^{>0}_R(\syz^d_Rk,C)=0$, we get an exact sequence 
$$
0\to \Hom_R(\syz^d_R k, C)\to C^{\oplus b_d}\to \cdots \to C^{\oplus b_{n-1}}\to \Hom_R(\syz^n_R k, C)\to 0.
$$
Since $\cx_R(C)<\infty$ and $\cx_R (\Hom_R(\syz^n_Rk,C))<\infty$ by hypothesis, we get $\cx_R (\Hom_R(\syz^d_Rk,C))<\infty$. Therefore $R$ is a complete intersection by part (1).  Finally, by replacing $d$ by $n$ in the last part of the argument for (1), it similarly follows that  $\cx_R(\Hom_R(\syz^n_R k, C))=\cx_R(k)=\codim R.$
\end{proof}  

\begin{cor}\label{trnew} Let $(R,\m,k)$ be a local ring. If there exists an integer $n\ge \depth R$ such that  $\rpdim_R(\Tr_R \syz^n_R k)\le 2$, then $R$ is a complete intersection of codimension at most two.  
\end{cor}

\begin{proof} The hypothesis along with Theorem \ref{Gdimcor} implies that $R$ is Gorenstein. Let $d=\dim R$. There exists a free module $F$ such that $\syz^2\Tr_R \syz^n_R k \oplus F \cong  (\syz^{n}_R k)^*$. By \cite[Proposition 3.8 and Corollary 3.13]{redring}, we have $\rpdim_R((\syz^{n}_R k)^*) = \rpdim_R(\syz^2\Tr_R \syz^n_R k) = \rpdim_R(\Tr_R \syz^n_R k)\le 2$. Since $R$ is Gorenstein and $n\ge \dim R$, $(\syz^{n}_R k)^*$ is  totally reflexive, and thus $\cx_R((\syz^n_R k)^*)\le 2$ or $\cx_R(\syz^n_R k)\le 2$ holds by Proposition \ref{c1}. Applying Proposition \ref{l4} (with $R=C$), we conclude that $R$ is a complete intersection of codimension less than or equal to $2$.
\end{proof}

\begin{cor} \label{t1} Let $(R, \fm, k)$ be a Gorenstein local ring of dimension $d$. Then, the following are equivalent:

\begin{enumerate}[\rm(1)]
    \item $R$ is a complete intersection of codimension at most two.
    
    \item For every integer $n\ge 1$, $\rpdim_R(\syz^n_Rk)\le 2$.  
    
    \item There exists an integer $n\ge 1$ such that $\rpdim_R(\syz^n_Rk)\le 2$. 
\end{enumerate}   
\end{cor}

\begin{proof} $(1)\implies (2)$ Follows from Corollary \ref{corcxrpd}. $(2)\implies (3)$ is obvious. 
$(3)\implies (1)$ Using \cite[Corollary 3.13]{redring} we can pass to a higher syzygy to assume that there exists an integer $n\ge d$ such that $\rpdim_R(\syz^n_Rk)\le 2$. Then $\syz^n_R k$ is totally reflexive, and therefore proposition \ref{c1} shows that either $\cx (\syz^n_R k)\le 2$ or $\cx ((\syz^n_R k)^*) \le 2$.
Then Proposition \ref{l4} implies that $R$ is a complete intersection of codimension less than or equal to $2$.    
\end{proof} 


We are now ready to give a proof of $c=2$ case of Theorem \ref{mainci}:     

\begin{proof} [Proof of $c=2$ case of Theorem \ref{mainci}]  Note that (ii) $\Longrightarrow$ (iii) follows from Theorem \ref{Gdimcor} and (iii)  $\Longrightarrow$ (i) is due to Corollary \ref{t1}. Since (i) $\Longrightarrow$ (ii) is clear from Corollary \ref{corcxrpd},  this proves $\rm{(i)}\iff \rm{(ii)}\iff \rm{(iii)}$. Finally, (iv)$\iff$(i) follows by Corollaries \ref{corcxrpd} and  \ref{trnew}.
\end{proof}

Now we finally give a proof of Theorem \ref{corcx}. We start with an argument on sequences of positive real numbers.

\begin{chunk} 
Let $c \ge 0$ be an integer. We say that a sequence $\{a_n\}_{n=0}^\infty$ of real numbers has {\it polynomial growth of degree $c-1$} if it is eventually non-negative and there is a real number $A >0$ such that $a_n \le A n^{c-1}$ for all $n \gg 0$. Note that if $\{a_n\}_{n=0}$ has polynomial growth of degree $c-1$, then $\underset{n\to \infty}{\lim}\, a_n/n^{c} = 0$.
For example, for an $R$-module $M$, $\cx(M) \le c$ if and only if the Betti sequence $\{\beta_{n}^R(M)\}_{n=0}^\infty$ has polynomial growth of degree $c-1$. The following lemma is elementary.
\end{chunk}

\begin{lem}\label{poly} Let $c\geq 0$ be an integer, $f(t)\in \bR[t]$ be a polynomial of degree $c-1$ with a positive leading term and let $u$ be a real number such that $0 < u \le 1$.
We consider the sequence $\{a_n\}_{n=0}^\infty$ given by $a_n = \sum_{l=0}^n u^{n-l} f(l)$.
Then $\{a_n\}_{n=0}^\infty$ has polynomial growth of degree $c$. Furthermore, if $u < 1$, then $\{a_n\}_{n=0}^\infty$ has polynomial growth of degree $c-1$. 
\end{lem} 

\begin{proof}
The case of $u=1$ is easy. We shall assume $u< 1$ and proceed by induction on $c$ to show $\{a_n\}_{n=0}^\infty$ has polynomial growth of degree $c-1$. If $c=0$, there is nothing to prove. Consider the case of $c \ge 1$. Then 
$$
a_n - ua_n = \sum_{l=0}^n u^{n-l} f(l) - u\sum_{l=0}^n u^{n-l} f(l) = - u^{n+1} f(0) - \sum_{l=0}^{n-1} u^{n-i}(f(l+1) -f (l)) + f(n)
$$
and hence
\begin{equation} \tag{\ref{poly}.1}
a_n = \dfrac{1}{1-u} \left\{ - u^{n+1} f(0) - \sum_{l=0}^{n-1} u^{n-i}(f(l+1) -f (l)) + f(n)\right\}.  
\end{equation}
Because $f(t+1)- f(t)$ is a polynomial with a positive reading term, the induction hypothesis shows that $\{\sum_{l=0}^{n-1} u^{n-i}(f(l+1) -f (l))\}_{n=0}^\infty$ has polynomial growth of degree $c-2$.
Thus, (\ref{poly}.1) shows $a_n \le \dfrac{1}{1-u}(f(n) + 1)$ for all $n \gg 0$. Moreover, $\underset{n\to \infty}{\lim}\, a_n/n^{c-1} = \underset{n\to \infty}{\lim}\, f(n)/(1-u)n^{c-1} > 0$ yields that $\{a_n\}_{n=0}^\infty$ is eventually non-negative.
Therefore, we are done.
\end{proof}

Now, we are ready to prove Theorem \ref{corcx}.

\begin{proof}[Proof of Theorem \ref{corcx}] Since $\beta_i(K)\le a\beta_{i+n}(M)+b\beta_i(M)$, $\cx_R(K)\le \cx_R(M)$. Take a real number $A>0$ such that $\beta_i(K) \le A i^{c-1}$ for all $i\ge 1$.
Set $u:=a/b$. 

First, consider the case of $a \le b$ i.e., $u \le 1$.
From the short exact sequence, there are inequalities
$$
b \beta_{i+n+1}(M) \le \beta_{i+1}(K) + a\beta_{i}(M) \le A (i+1)^{c-1} + a\beta_{i}(M)
$$
and hence we  get 
\begin{equation}\tag{\ref{corcx}.1}\label{eq1}
\beta_{i+n+1}(M) \le \frac{A}{b} (i+1)^{c-1} + u\beta_{i}(M)
\end{equation}
for all $i \ge 1$.

\begin{claim}\label{claim1}
For integers $r,i\ge 1$, we have an inequality
$$
\beta_{i+r(n +1)}(M) \le \frac{A}{b}\sum_{l=0}^{r-1}u^{r-l-1} (i+l(n+1)+1)^{c-1} + u^{r} \beta_{i}(M).
$$	
\end{claim}

\begin{proof}[Proof of Claim $1$]
We use the induction on $r$.
The case of $r=1$ is nothing but the inequality (\ref{eq1}).
Assume $r \ge 2$ and the inequality
$$
\beta_{i+(r-1)(n +1)}(M) \le \frac{A}{b}\sum_{l=0}^{r-2}u^{r-l-2} (i+l(n+1)+1)^{c-1} + u^{r-1} \beta_{i}(M)
$$
holds true. Then we use the inequality (\ref{eq1}) again, we obtain
\begin{align*}
\beta_{i+r(n +1)}(M) &\le \frac{A}{b} (i+r(n+1) -n)^{c-1} + u\beta_{i+(r-1)(n+1)}(M) \\
&= \frac{A}{b} (i+(r-1)(n+1) +1)^{c-1} + u\beta_{i+(r-1)(n+1)}(M) \\
&\le \frac{A}{b} (i+(r-1)(n+1) +1)^{c-1} + u\left\{\frac{A}{b}\sum_{l=0}^{r-1}u^{r-l-2} (i+l(n+1)+1)^{c-1} + u^{r-1} \beta_{i}(M) \right\} \\
&= \frac{A}{b}\sum_{l=0}^{r-1}u^{r-l-1} (i+l(n+1)+1)^{c-1} + u^{r} \beta_{i}(M),
\end{align*}
where the second inequality uses the induction hypothesis.
\end{proof}

Set $B:= \max\{\beta_i(M) \mid i=1,2,\ldots,n+1\}$.
For an integer $m \ge 2(n+1)$, write $m=r(n+1) + i$ for some $i \in \{1,2,\ldots, n+1\}$ and $r \ge 1$. Then one has
$$
\beta_{m}(M) \le \frac{A}{b}\sum_{l=0}^{r-1}u^{r-l-1} (i+l(n+1)+1)^{c-1} + u^{r} \beta_{i}(M) \le \frac{A}{b}\sum_{l=0}^{r-1}u^{r-l-1} (i+l(n+1)+1)^{c-1} + B.
$$

Applying Lemma \ref{poly} to the first term, we conclude the statements $(1)$ and $(2)$.

Assume that $a > b$ i.e., $u > 1$.
From the short exact sequence, there are inequalities 
$$
a\beta_i(M) \le \beta_i(K) + b\beta_{i+n+1}(M) \le Ai^{c-1} + b\beta_{i+n+1}(M)
$$ 
and hence we get 
\begin{equation}\tag{\ref{corcx}.2}\label{eq2}
\beta_{i+n+1}(M) \ge u\beta_i(M) - \dfrac{A}{b}i^{c-1}
\end{equation}
for all $i \ge 1$.

\begin{claim}\label{claim2}
For any integers $r,i\ge 1$, we have an inequality
$$
\beta_{i+r(n+1)}(M)	\ge u^r \left\{\beta_i(M) - \frac{A}{b}i^{c-1} \sum_{l=0}^{r-1}u^{-l-1}\left(1+\dfrac{l(n+1)}{i}\right)^{c-1} \right\}.
$$
\end{claim}

\begin{proof}[Proof of Claim $2$]
As in the proof of Claim \ref{claim1}, we use the induction on $r$.
The inequality holds if $r = 1$ by (\ref{eq2}).
Consider the case of $r \ge 2$ and assume 
$$
\beta_{i+(r-1)(n+1)}(M)	\ge u^{r-1} \left\{\beta_i(M) - \frac{A}{b} i^{c-1}\sum_{l=0}^{r-2}u^{-l-1}\left(1+\dfrac{l(n+1)}{i}\right)^{c-1} \right\}.
$$
Then using (\ref{eq2}), we obtain
\begin{align*}
\beta_{i+r(n+1)}(M)	&\ge u \beta_{i+(r-1)(n+1)}(M) - \frac{A}{b} \left(i+(r-1)(n+1)\right)^{c-1} \\
&\ge u\left[ u^{r-1}\left\{\beta_i(M) - \frac{A}{b}i^{c-1} \sum_{l=0}^{r-2}u^{-l-1}\left(1+\dfrac{l(n+1)}{i}\right)^{c-1}\right\}\right] - \frac{A}{b} \left(i+(r-1)(n+1)\right)^{c-1}  \\
&\ge u^r \left\{\beta_i(M) - \frac{A}{b}i^{c-1} \sum_{l=0}^{r-1}u^{-l-1}\left(1+\dfrac{l(n+1)}{i}\right)^{c-1} \right\},
\end{align*}
where the second inequality follows from the induction hypothesis.
\end{proof}

We note that for any integers $r,i \ge 1$ the following inequalities hold:
$$
\sum_{l=0}^{r-1}u^{-l-1}\left(1+\dfrac{l(n+1)}{i}\right)^{c-1} \le \sum_{l=0}^{r-1}u^{-l-1}\left(1+l(n+1)\right)^{c-1} \le \sum_{l=0}^{\infty}u^{-l-1}\left(1+l(n+1)\right)^{c-1},
$$
where the last sum $C := \sum_{l=0}^{\infty}u^{-l-1}\left(1+l(n+1)\right)^{c-1}$ is finite as $u > 1$.
Then we get inequalities
\begin{align}\tag{\ref{corcx}.3}\label{eq3}
\begin{aligned}
\beta_{i+r(n+1)}(M)	&\ge u^r \left\{\beta_i(M) - \frac{A}{b} i^{c-1}\sum_{l=0}^{r-1}u^{-l-1}\left(1+\dfrac{l(n+1)}{i}\right)^{c-1} \right\} \\
& \ge u^{r}\left\{\beta_i(M) - \dfrac{AC}{b} i^{c-1}\right\}
\end{aligned}
\end{align}
by Claim \ref{claim2} for any integers $r, i \ge 1$.

Now, we assume $\cx_R(M)\ne \cx_R(K)$. As $c=\cx_R(K)\le \cx_R(M)$, so then $\cx_R(M) > c$ and we now prove $\cx_R(M) = \infty$.  
Then there is a sufficiently large integer $k$ such that  
$
\beta_{k}(M) > \frac{C}{b} k^{c-1} +1.
$
Combining this inequality with (\ref{eq3}), we conclude that there is an integer $k$ such that $\beta_{k+r(n+1)}(M) \ge u^r$ for any integer $r\ge 1$. This means that $\cx_R(M)=\infty$. 
\end{proof}
   
\begin{rem} Let $R$ be a local ring and let $M$ and $N$ be $R$-modules. If
$f_R(M, N) := \inf \{j \ge 0 \mid l_R(\Ext_R^i(M, N)) < \infty \mbox{ for all } i \ge j\} < \infty$, then one can define
$$\beta^i_R(M, N)= l_R(\Ext_R^i(M, N)) \text{ for all } i \ge f_R(M, N)$$ 
and
$$
\cx_R(M, N) = \inf \{c \ge 0 \mid \exists r> 0 \mbox{ s.t. } \beta^i_R(M, N) \le r i^{c-1}\,\, (\forall i \ge f_R(M, N))\} 
$$ 
Note, If $N = k$, then $f_R(M, k) = 0$, $\beta_R^i(M, k) = \beta^i_R(M)$ is the  Betti number of $M$, and $\cx_R(M, k) = \cx_R(M)$ is the  complexity of $M$.  

An argument verbatim to that of the proof of Theorem \ref{corcx} also establishes the following result:  

\label{redcx} Let 
$$
0 \to M^{\oplus a} \to K \to \Omega^n_R M^{\oplus b} \to 0  
$$
be a short exact sequence, $X$ an $R$-module such that $f:=f_R(M, X) < \infty$ and $\cx_R(K, X) =c < \infty$.
\begin{enumerate}[\rm(1)]
\item
If $a < b$, then $\cx_R(M, X) = \cx_R(K, X)$.
\item	
If $a = b$, then $\cx_R(M, X)\le \cx_R(K, X) +1$.
\item
If $a > b$, then either $\cx_R(M, X)= \cx_R(K, X)$ or $\cx_R(M, X)=\infty$. 
\end{enumerate}

Applying this to $X=R$ and $M$ locally free on punctured spectrum, one sees that $\cx_R(M,R)=\gcx_R(M)$ is exactly the Gorenstein complexity, as in the notation of \cite[Definition 4.1]{at}. In this case, the above inequalities give relations between Gorenstein complexities of modules that fit into a short exact sequences, and the Gorenstein complexity and $\rGdim$ version of Corollary \ref{corcxrpd} holds (see \cite[Proposition 4.3]{at}). 
\end{rem}

\section*{acknowledgements}
The authors thank Jesse Cook for his comments, suggestions, and help during the preparation of this manuscript. The authors also thank Tokuji Araya for his help and discussions on earlier versions of this manuscript.



\end{document}